\documentclass[11pt]{amsart}
\usepackage{amssymb,color}

\newcommand{\cov}{\mathrm{cov}}

\newcommand{\M}{\mathcal M}
\newcommand{\C}{\mathcal C}
\newcommand{\IN}{\mathbb N}
\newcommand{\IR}{\mathbb R}

\newcommand{\w}{\omega}
\newcommand{\e}{\varepsilon}
\newcommand{\defeq}{\overset{\mbox{\tiny\sf def}}=}
\newcommand{\Ball}{\mathsf B}
\newcommand{\II}{\mathbb I}
\newcommand{\A}{\mathcal A}
\newcommand{\U}{\mathcal U}
\newcommand{\V}{\mathcal V}

\newcommand{\pr}{\mathrm{pr}}

\newcommand{\supp}{\mathrm{supp}}

\newtheorem{theorem}{Theorem}
\newtheorem{proposition}{Proposition}
\newtheorem{claim}{Claim}
\newtheorem{problem}{Problem}
\newtheorem{lemma}{Lemma}
\newtheorem{corollary}{Corollary}
\theoremstyle{definition}

\newtheorem*{definition*}{Definition}
\newtheorem{example}{Example}

\title{Polyboundedness of zero-closed semigroups}
\author{Taras Banakh and Andriy Rega}
\address{Jan Kochanowski University in Kielce (Poland), and Ivan Franko National University of Lviv (Ukraine)}
\email{t.o.banakh@gmail.com, andrijrega@gmail.com}

\begin{document}
\begin{abstract}
The {\em polyboundedness number} $\cov(\A_X)$ of a semigroup $X$ is the smallest cardinality of a cover of $X$ by sets of the form $\{x\in X:a_0xa_1\cdots xa_n=b\}$ for some $n\ge 1$, $b\in X$ and $a_0,\dots,a_n\in X^1=X\cup\{1\}$. Semigroups with finite polyboundedness number are called {\em polybounded}. A semigroup $X$ is called {\em zero-closed} if $X$ is closed in its $0$-extension $X^0=\{0\}\cup X$ endowed with any Hausdorff semigroup topology. We prove that any zero-closed infinite semigroup $X$ has $\cov(\A_X)<|X|$. Under Martin's Axiom, a zero-closed semigroup is polybounded if $X$ admits a compact Hausdorff semigroup topology or $X$ has a separable complete subinvariant metric.
\end{abstract}
\maketitle

For a semigroup $X$ its
\begin{itemize}
\item {\em $0$-extension} is the semigroup $X^0=X\cup\{0\}$ where $0\notin X$ is any element such that $0x=0=x0$ for every $x\in X^0$;
\item {\em $1$-extension} is the semigroup $X^1=X\cup\{1\}$ where $1\notin X$ is any element such that $1x=x=x1$ for every  $x\in X^1$.
\end{itemize}

Following \cite{CCCS}, we call a semigroup $X$ {\em zero-closed} if $X$ is closed in its $0$-extension $X^0=\{0\}\cup X$ endowed with an arbitrary Hausdorff semigroup topology. A topology $\tau$ on a semigroup $S$ is called a {\em semigroup topology} if it makes the binary operation $S\times S\to S$, $(x,y)\mapsto xy$, of $S$ continuous.

The zero-closedness is the weakest property among the categorical closedness property, considered in \cite{Ban}, \cite{CCUS}, \cite{CCCS}, \cite{ICT1S}, \cite{BB}, \cite{CCCS}, \cite{ACS}.

In this paper we find conditions guaranteeing that a zero-closed semigroup is polybounded. Polybounded semigroups were introduced and studied in \cite{CCCS}, where they were defined with the help of semigroup polynomials.

A {\em semigroup polynomial} on a semigroup $X$ is a function $f:X\to X$ of the form $$f(x)=a_0xa_1\cdots xa_n$$ for some $n\in\IN$ and $a_0,\dots,a_n\in X^1$. The elements $a_0,\dots,a_n$ are called the coefficients of the polynomial $f$, and the number $n$ is called the {\em degree} of the polynomial $f$ and is denoted by $\deg(f)$.

Following \cite{CCCS}, we define a semigroup $X$ to be {\em polybounded} if $X=\bigcup_{i=1}^n f_i^{-1}(b_i)$ for some elements $b_1,\dots,b_n\in X$ and some semigroup polynomials $f_1,\dots,f_n$ on $X$. The smallest possible number $n$ in this definition in called the {\em polyboundedness number} of $X$.

In fact, the polyboundedness number can be equivalently defined for any semigroup $X$ as the covering number of the family $\A_X$ of algebraic sets in $X$. A subset $A$ of a semigroup $X$ is called {\em algebraic} if $$A=\{x\in X:a_0xa_1\cdots xa_n=b\}$$ for some number $n\in\IN$ and elements
$b\in X$ and $a_0,\dots,a_n\in X^1$.

Let $\A_X$ be the family of all algebraic sets in a semigroup $X$. For a subset $S\subseteq X$, the cardinal
$$\textstyle\cov(S;\A_X)\defeq\{|\C|:\C\subseteq \A,\;\;S\subseteq \bigcup\C\}$$ is called the {\em polyboundedness number} of $S$ in $X$. It is equal to the smallest cardinality of a cover of $S$ by algebraic subsets of $X$. If $S=X$, then we shall write $\cov(\A_X)$ instead of $\cov(X;\A_X)$. A subset $S$ of a semigroup $X$ is called {\em polybounded in} $X$ if $\cov(S;\A_X)$ is finite.

Observe that a semigroup $X$ is polybounded if and only if $\cov(\A_X)<\w$.
Polybounded semigroups play an important role in the theory of categorically closed semigroups \cite{CCCS}, \cite{ACS} and have many nice properties. For example, cancellative polybounded semigroups are groups \cite[1.8]{CCCS}, polybounded paratopological groups are topological groups \cite[8.1]{CCCS}, each polybounded group $X$ is closed in any $T_1$ topological semigroup containing $X$ as a discrete subsemigroup \cite[6.1]{CCCS}. The polyboundedness number of a (topologized) semigroup is studied in \cite{BR}.

In \cite[Theorem 6.4]{CCCS} Banakh and Bardyla proved that every zero-closed countable semigroup is polybounded. The following theorem generalizes this results to semigroups of arbitrary cardinality.

\begin{theorem}\label{t:poly} Every zero-closed semigroup $X$ has polyboundedness number $\cov(\A_X)<\max\{2,|X|\}$.
\end{theorem}

In light of Theorem~\ref{t:poly} it should be mentioned that the cardinal $\cov(\A_X)$ can be quite close to the cardinality of $X$: by \cite{BanS}, for every infinite cardinal $\kappa$ with $\kappa^+=2^\kappa$ there exists a zero-closed group $X$ of cardinality $|X|=\kappa^+$ with $\mathrm{cf}(\kappa)\le\cov(\A_X)\le\kappa<|X|$.
\smallskip

A subset $I$ of a semigroup $X$ is called an {\em ideal} if $IX\cup XI\subseteq I$. An ideal $I$ in $X$ is called {\em minimal} if $I\ne \emptyset$ and $I\subseteq J$ for any nonempty ideal $J$ in $X$. A minimal ideal if exists, is unique, and is called {\em the minimal ideal} of $X$. A semigroup $X$ has a minimal ideal if and only if $\bigcap_{a\in X}XaX\ne\emptyset$, in which case $\bigcap_{a\in X}XaX$ is the minimal ideal of $X$.

\begin{proposition}\label{p:ideal} Every  zero-closed semigroup $X$ has a minimal ideal.
\end{proposition}

\begin{proof} Assuming that $\bigcap_{a\in X}XaX=\emptyset$, we conclude that $X$ is not closed in $X^0$ endowed with the Hausdorff semigroup topology $\tau$, generated by the base $\mathcal B=\big\{\{x\}:x\in X\big\}\cup\big\{\{0\}\cup XaX:a\in X\big\}$.
\end{proof}

For countable semigroups, Theorem~\ref{t:poly} (and also Theorem 6.4 of \cite{CCCS}) can be deduced from the following ``local'' characterization of polybounded zero-closed semigroups.

\begin{theorem}\label{t:local} A zero-closed semigroup $X$ is polybounded if and only if it admits a separable subinvariant pseudometric $d$ such that for some $o\in I\defeq \bigcap_{a\in X}XaX$ and $\e>0$ the set $\{x\in I:d(x,o)<\e\}$ is polybounded in $X$.
\end{theorem}

Let us recall that a {\em pseudometric} on a set $X$ is any function $d:X\times X\to\IR$ such that $d(x,x)=0$, $d(x,y)=d(y,x)$ and $d(x,z)\le d(x,y)+d(y,z)$ for all $x,y,z\in X$. A pseudometric $d$ on $X$ is {\em separable} if there exists a countable set $C\subseteq X$ such that for every $x\in X$ and $\e>0$ there exists $c\in C$ such that $d(x,c)<\e$. 

A pseudometric $d$ on a semigroup $X$ is {\em subinvariant} if $d(axb,ayb)\le d(x,y)$ for any $x,y\in X$ and $a,b\in X^1$.


Theorems~\ref{t:poly}, \ref{t:local} will be applied in the proof of the following theorem holding under the set-theoretic assumption $\cov(\M)=\mathfrak c$, where $\cov(\M)$ is the smallest cardinality of a cover of the real line by meager sets. The Baire Theorem guarantees that the cardinal $\cov(\M)$ is uncountable. Under Martin's axiom, $\cov(\M)$ is equal to $\mathfrak c$, the cardinality of the continuum. By Theorem~7.13 in \cite{Blass}, the equality $\cov(\M)=\mathfrak c$ is equivalent to Martin's Axiom for countable posets.

\begin{theorem}\label{t:MA} Under $\cov(\M)=\mathfrak c$, a zero-closed semigroup $X$ is polybounded if one of the following conditions is satisfied:
\begin{enumerate}
\item $X$ has a separable complete subinvariant metric;
\item $X$ admits a compact Hausdorff semigroup topology.
\end{enumerate}
\end{theorem}

Theorem~\ref{t:MA} suggests the following two open problems.

\begin{problem}\label{prob:main} Is every zero-closed Polish group $X$ polybounded (under $\cov(\M)=\mathfrak c$)?
\end{problem}

\begin{problem}\label{prob:comp} Is every zero-closed compact Hausdorff topological group polybounded in ZFC?
\end{problem}

A weaker property compared to the polyboundedness is the polyfiniteness.

A semigroup $X$ is defined to be {\em polyfinite} if there exist $n\in\IN$ and a finite set $F\subseteq X$ such that for any elements $x,y\in X$ there exists a semigroup polynomial $f:X\to X$ of degree $\le n$ such that $\{f(x),f(y)\}\subseteq F$.

In Lemma~\ref{l:pb=>pf} we shall prove that each polybounded semigroup is polyfinite. Polyfinite semigroups were introduced and used in \cite{ACS}. An example of a polyfinite group which is not polybounded is presented in Example~\ref{ex:AR}.

Let $\cov(\overline{\mathcal N})$ be the smallest cardinality of the cover of the real line by closed subsets of Lebesgue measure zero. By \cite[4.1]{BS},
$$\max\{\cov(\M),\cov(\mathcal N)\}\le\cov(\overline{\mathcal N})\le\max\{\cov(\mathcal N),\mathfrak d\}\le\mathfrak c,$$where $\cov(\mathcal N)$ is the smallest cardinality of a cover of the real line by subsets of Lebesgue measure zero and $\mathfrak d$ is the smallest cardinality of a cover of the space $\mathbb R\setminus\mathbb Q$ by compact subsets. By \cite[5.6]{BS}, the strict inequality $\max\{\cov(\M),\cov(\mathcal N)\}<\cov(\overline{\mathcal N)}$ holds in some models of  ZFC.

\begin{theorem}\label{t:pf} Under $\cov(\overline{\mathcal N})=\mathfrak c$, each zero-closed compact Hausdorff topological semigroup is polyfinite.
\end{theorem}

Theorem~\ref{t:pf} suggests the following weaker version of Problem~\ref{prob:comp}.

\begin{problem} Is every zero-closed compact Hausdorff topological group polyfinite?
\end{problem}

A topological semigroup $X$ is called {\em $0$-discrete} if $X$ is has a unique non-isolated point $0\in X$ such that $0x=0=x0$ for all $x\in X$.

Let $\C$ be a class of topological semigroups. A semigroup $X$ is called
\begin{itemize}
\item {\em absolutely $\C$-closed} if for any homomorphism $h:X\to Y$ to a topological semigroup $Y\in\C$ the image $h[X]$ is closed in $Y$.
\item {\em injectively $\C$-closed} if for any injective homomorphism $h:X\to Y$ to a topological semigroup $Y\in\C$ the image $h[X]$ is closed in $Y$.
\end{itemize}

Absolutely and injectively $\C$-closed semigroups for various classes $\C$ are studied in \cite{CCUS}, \cite{ICT1S}, \cite{ICT2S}, \cite{ACS}.

Applying Theorems~\ref{t:poly} and \ref{t:MA}, we shall prove the following corollaries that will be used in \cite{ACS}.

\begin{corollary}\label{c:IC-M} Let $\C$ be a class of topological semigroups, containing all $0$-discrete Hausdorff topological semigroups and all Polish topological semigroups. Let $X$ be an injectively $\C$-closed semigroup admitting a separable subinvariant metric. If $\cov(\M)=\mathfrak c$, then $X$ is polybounded.
\end{corollary}

\begin{corollary}\label{c:IC-C} Let $\C$ be a class of topological semigroups, containing all $0$-discrete Hausdorff topological semigroups, and all compact Hausdorff topological semigroups. Let $X$ be an injectively $\C$-closed subsemigroup of a compact Hausdorff topological semigroup.
\begin{enumerate}
\item If $\cov(\M)=\mathfrak c$, then $X$ is polybounded.
\item If $\cov(\overline{\mathcal N})=\mathfrak c$, then $X$ is polyfinite.
\end{enumerate}
\end{corollary}

A topological space $X$ is called
\begin{itemize}
\item {\em Lindel\"of} if every open cover of $X$ has a countable subcover;
\item {\em power-Lindel\"of} if every finite power $X^n$ of $X$ is Lindel\"of.
\end{itemize}
A topology $\tau$ on a set $X$ is called {\em power-Lindel\"of\/} if so is the topological space $(X,\tau)$.
For example, any Polish or compact topology is power-Lindel\"of. 
A topological space is {\em Tychonoff\/} if it is homeomorphic to a subspace of some Tychonoff cube $[0,1]^\kappa$.

A topology $\tau$ on a semigroup $X$ is called  {\em subinvariant} if for any distinct points $x,y\in X$ there exists a $\tau$-continuous subinvariant pseudometric $d$ on $X$ such that $d(x,y)>0$.

\begin{corollary}\label{c:AC} Assume that $\cov(\M)=\mathfrak c$. Let $\C$ be a class of topological semigroups, containing all Tychonoff topological semigroups. An absolutely $\C$-closed semigroup $X$ is polybounded if $X$ admits a power-Lindel\"of subinvariant topology.
\end{corollary}

Theorems~\ref{t:poly}, \ref{t:local}, \ref{t:MA}, \ref{t:pf} and Corollaries \ref{c:IC-M}, \ref{c:IC-C}, \ref{c:AC} will be proved in Sections~\ref{s:poly}, \ref{s:local}, \ref{s:MA},  \ref{s:pf}, \ref{s:IC-M}, \ref{s:IC-C}, \ref{s:AC}, respectively. In Sections~\ref{s:lim} and \ref{s:polyfinite} we prove Theorems~\ref{t:lim} and \ref{t:lim-polyfinite} on the preservation of polyboundedness and polyfiniteness by limits of inverse spectra. These theorems will be applied in the proofs of Theorems~\ref{t:MA}(2), \ref{t:pf} and Corollary~\ref{c:AC}.

\section{Proof of Theorem~\ref{t:local}}\label{s:local}

The ``only if'' part of Theorem~\ref{t:local} is trivial (just take the zero pseudometric on $X$). To prove the ``if'' part, assume that $X$ is a zero-closed semigroup and $d:X\times X\to X$ is a separable subinvariant pseudometric on $X$ such that for some point $o\in I\defeq\bigcap_{a\in X}XaX$ and $\epsilon>0$ the set $\{x\in I:d(x,o)<\epsilon\}$ is polybounded in $X$.
Replacing the pseudometric $d$ by $\frac1{\epsilon} d$, we can assume that $\epsilon=1$ and the set $\{x\in I:d(x,o)<1\}$ is polybounded in $X$.

For a point $x\in X$ and positive real number $\e$ we denote by $\Ball_d(x,\e)\defeq\{y\in X:d(x,y)<\e\}$ the $\e$-ball centered at $x$. For a subset $A\subseteq X$ let $\Ball_d(A,\e)\defeq\bigcup_{a\in A}\Ball_d(a,\e)$ be the $\e$-neighborhood of $A$ in the pseudometric space $(X,d)$.

To derive a contradiction, assume that the semigroup $X$ is not polybounded.

For every $n\in\IN$ and $\e>0$ let
$$\textstyle\{0,1\}^n_+\defeq \{s\in\{0,1\}^n:\sum_{i\in n}s(i)>0\}=\{0,1\}^n\setminus\{0\}^n.$$

For a nonzero number $n\in\IN$ and finite sequences $a=(a_0,\dots,a_{n})\in X^{n+1}$ and $s=(s_0,\dots,s_{n-1})\in\{0,1\}^{n}_+$ let $f_{a,s}:X\to X$ be the semigroup polynomial defined by $$f_{a,s}(x)=a_0x^{s_0}a_1x^{s_1}\dots a_{n-1}x^{s_{n-1}a_n}\quad\mbox{for $x\in X$,}
$$ where $x^1=x$ and $x^0=1\in X^1$ for every $x\in X$.

\begin{claim}\label{cl:step} For any finite sets $A\subseteq \bigcup_{n\in\IN}(X^{n+1}\times\{0,1\}^n)$ and $B\subseteq I$ there exists $x\in I$ such that $f_{a,s}(x)\notin \Ball_d(b,1)$ for every $(a,s)\in A$ and $b\in B$.
\end{claim}

\begin{proof} To derive a contradiction, assume that $I\subseteq \bigcup_{(a,s)\in A}\bigcup_{b\in B}f_{a,s}^{-1}(\Ball_d(b,1))$.  Since $I$ is a minimal ideal in $X$, for every $b\in B$ we have $o\in I=IbI$ and hence there exist elements $u_b,v_b\in I$ such that $o=u_bbv_b$. The elements $u_b,v_b$ determine the semigroup polynomial $g_b:X\to X$, $g_b:x\mapsto u_bxv_b$, such that $g_b(b)=o$ and $g_b[X]=u_bXv_b\subseteq IXI\subseteq I$.

By the choice of the point $o$, the set $I\cap\Ball_d(o,1)$ is polybounded in $X$. So, there exist a finite set $Q\subseteq X$ and a finite set $P$ of semigroup polynomials on $X$ such that $I\cap\Ball_d(0,1)\subseteq\bigcup_{p\in P}\bigcup_{q\in Q}p^{-1}(q)$. For every $(a,s)\in A$, $b\in B$ and $p\in P$, consider the semigroup polynomial $h_{a,s,b,p}= p\circ g_b\circ f_{a,s}\circ g_b:X\to X$.
For every $x\in X$ we have $g_b(x)\in I$ and then $f_{a,s}(g_b(x))\in \Ball_d(b,1)$ for some $(a,s)\in A$ and $b\in B$. Taking into account that the pseudometric $d$ is subinvariant, we conclude that for the element $y=f_{a,s}(g_b(x))\in \Ball_d(b,1)$ we have  $$d(g_b(y),o)=d(u_byv_b,u_bbv_b)\le d(y,b)<1$$and hence
$g_b(y)\in p^{-1}(q)$ for some $p\in P$ and $q\in Q$.
Then $h_{a,s,b,p}(x)=p\circ g_b\circ f_{a,s}\circ g_b(x)=p(g_b(y))=q$ and hence
$$x\in \bigcup_{(a,s)\in A}\bigcup_{b\in B}\bigcup_{p\in P}\bigcup_{q\in Q}h_{a,s,b,p}^{-1}(q),$$
witnessing that the semigroup $X$ is polybounded, which contradicts our assumption.
\end{proof}

Since the pseudometric $d$ is separable, there exists a sequence $\{c_n\}_{n\in\w}\subseteq X$ such that $X=\bigcup_{n\in\w}\Ball_d(c_n,\e)$ for every $\e>0$.

Let $x_0=c_0$ and $C_0=\{c_0\}$. Using Claim~\ref{cl:step}, for every $n\in\IN$ choose inductively a point $x_n\in I$ and a finite set $C_n\subseteq X$ such that the following conditions are satisfied:
\begin{enumerate}
\item $C_n=\{c_n\}\cup C_{<n}\cup\bigcup_{k=1}^n\{f_{a,s}(x_i):a\in C_{<n}^{k+1},\;s\in\{0,1\}^k_+,\; i<n\}$, where $C_{<n}=\bigcup_{k\in n}C_k$;
\item $x_n\in I\setminus\bigcup_{k=1}^n\bigcup_{a\in C_n^{k+1}}\bigcup_{s\in\{0,1\}^k_+}\bigcup_{c\in C_n}f_{a,s}^{-1}(\Ball_d(c,1))$.
\end{enumerate}

For every positive $\e\le 1$ let
$$\textstyle(0,1]^n_{<\e}\defeq\{x\in(0,1]^n:\sum_{i\in n}x(i)<\e\}$$
and
$$U(\e)\defeq\bigcup_{m\in\IN}\bigcup_{2\le k<\e m}\textstyle\big\{f_{a,s}(x_m):a\in\bigcup\big\{\prod_{i=0}^k\Ball_d(C_m,\delta(i)):\delta\in(0,1]_{<\e}^k\big\},\; s\in \{0,1\}_+^{k-1}\big\}.$$

The definition implies that $U(\e)\subseteq U(\e')$ for any positive real numbers $\e\le\e'\le 1$.

\begin{claim}\label{cl2} For every positive real number $\e\le\frac12$ we have $U(\e)U(\e)\subseteq U(2\e)$.
\end{claim}

\begin{proof} Given two elements $u_1,u_2\in U(\e)$, for every $i\in\{1,2\}$, find numbers $m_i\in\IN$, $k_i\in[2,\e m_i)$, $\delta_i\in (0,1]^{k_i}_{<\e}$, $a_i\in \prod_{j=0}^{k_i}\Ball_d(C_{m_i},\delta_i(j))$ and $s_i\in\{0,1\}^{k_i-1}_+$ such that $u_i=f_{a_i,s_i}(x_{m_i})$.

If $m_1=m_2$, then let $m=m_1=m_2$, $k=k_1+k_2$, and $\delta=\delta_1\hat{\;}\delta_2$, $a=a_1\hat{\;}a_2$, $s=s_1\hat{\;}(0)\hat{\;}s_2$ be the concatenations of the corresponding sequences. Observe that $4\le k=k_1+k_2<2\e m$, $\delta\in(0,1]^{k}_{<2\e}$, $a\in\prod_{i=0}^k\Ball_d(C_m,\delta(i))$,  $s\in\{0,1\}^{k-1}_+$, and $u_1u_2=f_{a,s}(x_m)$, witnessing that $u_1u_2\in U(2\e)$.

If $m_1<m_2$, then let $m=m_2$, $k=1+k_2<k_1+k_2<\e m_1+\e m_2<2\e m$, $\delta=(\sum_{j\in k_1}\delta_1(j))\hat{\;}\delta_2\in(0,1]^{k}_{<2\e}$, $a=(f_{a_1,s_1}(x_{m_1}))\hat{\;}a_2$, and $s=(0)\hat{\;}s_2\in\{0,1\}^{k-1}_+$.
We claim that $a(0)=f_{a_1,s_1}(x_{m_1})\in \Ball_d(C_m,\delta(0))$.
Since $a_1\in \prod_{i=0}^{k_1}\Ball_d(C_{m_1},\delta(i))$, there exists $a'_1\in C_{m_1}^{k_1}\subseteq C_{<m}^{k_1}$ such that $d(a_1(i),a_1'(i))<\delta_1(i)$ for all $i\in k_1<m_1\e<m$. By the condition (1) of the inductive construction, $f_{a_1',s_1}(x_{m_1})\in C_m$. The subinvariance of the pseudometric $d$ ensures that
$$d(f_{a_1,s_1}(x_{m_1}),f_{a_1',s_1}(x_{m_1}))\le\sum_{i\in k_1}d(a_1(i),a_1'(i))<\sum_{i\in k_1}\delta_1(i)=\delta(0)$$and hence $a(0)=f_{a_1,s_1}(x_{m_1})\in \Ball_d(C_m,\delta(0))$ and $a\in \prod_{i=0}^k\Ball_d(C_m,\delta(i))$. Then $u_1u_2=f_{a,s}(x_m)\in U(2\e)$ by the definition of $U(2\e)$.

By analogy we can show that $u_1u_2\in U(2\e)$ if $m_1>m_2$.
\end{proof}

\begin{claim}\label{cl3} For every $\e\in(0,1]$ and $b\in X$ there exists $\e'\in(0,1]$ such that $bU(\e')\cup U(\e')b\subseteq U(\e)$.
\end{claim}

\begin{proof} Find $\mu\in\IN$ such that $b\in \Ball_d(c_\mu,\frac\e2)$. Choose $\e'<\frac\e2$ such that $\e' \mu<2$. Given any $u\in U(\e')$, find $m\in\IN$, $k\in[2,\e' m)$, $\delta\in(0,1]^k_{<\e'}$, $a\in\prod_{i=0}^k\Ball_d(C_m,\delta(i))$, and $s\in\{0,1\}^k_+$ such that $u=f_{a,s}(x_m)$. It follows from $\e' \mu<2\le k<\e' m$ that $\mu<m$ and hence $c_\mu\in C_{m}$. Let $k'=1+k<2\e' m\le \e m$, $\delta'=(\frac12\e)\hat{\;}\delta\in(0,1]^{k'}_{<\frac12\e+\e'}\subseteq(0,1]^{k'}_{\e}$, $a'=(b)\hat{\;}a\in \prod_{i=0}^{k'}\Ball_d(C_m,\delta'(i))$, $s'=(0)\hat{\;}s\in\{0,1\}^{k'-1}_+$, and observe that $bu=f_{a',s'}(x_m)\in U(\e)$ and hence $bU(\e')\subseteq U(\e)$.

By analogy we can prove that $U(\e')b\subseteq U(\e)$.
\end{proof}

\begin{claim}\label{cl4} For every $b\in X$ there exists $\e\in(0,1]$ such that $b\notin U(\e)$.
\end{claim}

\begin{proof} Find $\mu\in\IN$ such that $b\in \Ball_d(c_\mu,\frac12)$. Choose $\e\in(0,\frac12]$ such that $\mu\e<2$. Assuming that $b\in U(\e)$, we can find $m\in\IN$, $k\in[2,\e m)$, $\delta\in(0,1]^k_{<\e}$, $a\in \prod_{i=0}^k\Ball_d(C_m,\delta(i))$, and $s\in\{0,1\}^{k-1}_+$ such that $b=f_{a,s}(x_m)$. Find $a'\in C_m^{k+1}$ such that $d(a(i),a'(i))<\delta(i)$ for all $i\le k$. The subinvariance of the pseudometric $d$ ensures that $d(f_{a,s}(x_m),f_{a',s}(x_m))\le\sum_{i=0}^kd(a(i),a'(i))<\sum_{i=0}^k\delta(i)<\e\le\frac12$.
Then $$d(f_{a',s}(x_m),c_\mu)\le d(f_{a',s}(x_m),f_{a,s}(x_m))+d(f_{a,s}(x_m),c_\mu)<\tfrac12+\tfrac12<1,$$which contradicts the inductive condition (2) as $\mu\e<2\le k<m\e$ and hence $\mu<m$ and $c_\mu\in C_m$.
\end{proof}

Consider the topology $\tau$ on $X^0$, generated by the base
$$\big\{\{x\}:x\in X\big\}\cup\big\{\{0\}\cup U(\e):\e\in(0,1]\big\}.$$
Claims~\ref{cl2}--\ref{cl4} ensure that $\tau$ is a Hausdorff semigroup topology on $X^0$. Since $0$ is a unique non-isolated point in $(X^0,\tau)$, the set $X$ is not closed in $(X^0,\tau)$, witnessing that the semigroup $X$ is not zero-closed.

\section{Proof of Theorem~\ref{t:poly}}\label{s:poly}

Let $X$ be a nonempty zero-closed semigroup.

If $X$ is finite, then by Theorem 2.7 \cite{HS}, $X$ contains an idempotent $e$ such that the subsemigroup $eXe$ of $X$ is a group. Since the group $eXe$ is finite, there exists $n\in\IN$ such that $x^n=e$ for all $x\in eXe$. Then for the semigroup polynomial $f:X\to X$, $f:x\mapsto (exe)^n$, we have $X=f^{-1}(e)$, which means that $\cov(\A_X)\le 1<\max\{2,|X|\}$.
\smallskip

Next, assume that $X$ is countable and infinite.

By Proposition~\ref{p:ideal}, the set $I\defeq\bigcap_{a\in X}XaX$ is non-empty. Fix any point $o\in I$. Observe that the $\{0,1\}$-valued metric
$d:X\times X\to \{0,1\}$ on $X$ is separable and subinvariant, and the unit ball $\Ball_d(o,1)=\{o\}$ is polybounded in $X$. By Theorem~\ref{t:local}, the semigroup $X$ is polybounded and $\cov(\A_X)<\w=\max\{2,|X|\}$.
\smallskip

Finally, assume that $X$ is uncountable.
To derive a contradiction, assume that $\cov(\A_X)\ge|X|$.

For a set $C\subseteq X$, let $P(C)$ be the set of all semigroup polynomials on $X$ with coefficients in $C$. It is clear that $|P(C)|\le\max\{|C|,\w\}$.
Write the set $X$ as $X=\{c_\alpha\}_{\alpha\in|X|}$.

By transfinite induction for every $\alpha<|X|$ choose a subsemigroup $C_\alpha\subset X^1$ and an element $x_\alpha\in X$ so that the following conditions are satisfied:
\begin{enumerate}
\item the semigroup $C_\alpha$ is generated by the set $\{1\}\cup\{c_\beta\}_{\beta\le\alpha}\cup\{x_\beta\}_{\beta<\alpha}$;
\item $f(x_\alpha)\notin C_\alpha$ for every semigroup polynomial $f\in P(C_\alpha)$.
\end{enumerate}

The element $x_\alpha$ exists since $\cov(\A_X)\ge|X|>|C_\alpha\times P(C_\alpha)|$.

For every $\alpha<|X|$ consider the set  $$U_{\alpha}=\bigcup_{\alpha\le \beta<|X|}\{f(x_\beta):f\in P(C_{\alpha})\}.$$

\begin{claim}\label{cl5} For any ordinals $\alpha\le\beta<|X|$ we have
\begin{enumerate}
\item $U_\beta\subseteq U_\alpha$;
\item $U_\alpha U_\alpha\subseteq U_\alpha$;
\item $C_\alpha U_\alpha C_{\alpha}\subseteq U_\alpha$;
\item $C_\alpha\cap U_\alpha=\emptyset$.
\end{enumerate}
\end{claim}

\begin{proof}
1. The inclusion $U_\beta\subseteq U_\alpha$ follows from $\alpha\le\beta$ and the definition of the sets $U_\alpha$ and $U_\beta$.
\smallskip

2. Given two elements $u_1,u_2\in U_\alpha$, for every $i\in\{1,2\}$, find an ordinal $\gamma_i\ge\alpha$ and a polynomial $f_i\in P(C_{\gamma_i})$ such that $u_i=f_i(x_{\gamma_i})$.

If $\gamma_1=\gamma_2$, then $u_1u_2=f(x_\gamma)$ for the ordinal $\gamma=\gamma_1=\gamma_2$ and the polynomial $f:X\to X$, $f:x\mapsto f_1(x)f_2(x)$. It is clear that $f\in P(C_\gamma)$ and hence  $u_1u_2\in U_\alpha$.

If $\gamma_1<\gamma_2$, then $u_1u_2=f(x_{\gamma_2})$ for the
 the polynomial $f:X\to X$, $f:x\mapsto f_1(x_{\gamma_1})f_2(x)$. It is clear that $f\in P(C_{\gamma_2})$ and hence  $u_1u_2\in U_\alpha$.

If $\gamma_1>\gamma_2$, then $u_1u_2=f(x_{\gamma_1})$ for the
 the polynomial $f:X\to X$, $f:x\mapsto f_1(x)f_2(x_{\gamma_2})$. It is clear that $f\in P(C_{\gamma_1})$ and hence  $u_1u_2\in U_\alpha$.
\smallskip

3. Given any elements $a,b\in C_\alpha$ and $u\in U_\alpha$, find an ordinal $\gamma\in [\alpha,|X|)$ and a semigroup polynomial $f\in P(C_\gamma)$ such that  $u=f(x_\gamma)$. Then $aub=g(x_\gamma)$ for the semigroup polynomial $g:X\to X$, $g:x\mapsto af(x)b$. It is clear that $g\in P(C_\gamma)$ and hence $aub\in U_\alpha$.
\smallskip

4. Assuming that $C_\alpha\cap U_\alpha$ contains some element $c$, find an ordinal $\gamma\in[\alpha,|X|)$ and a semigroup polynomial $f\in P(C_\gamma)$ such that $f(x_\gamma)=c$. But $f(x_\gamma)=c\in C_\alpha\subseteq C_\gamma$ contradicts the choice of $x_\gamma$.
\end{proof}

Let $\tau$ be the topology on $X^0=\{0\}\cup X$ generated by the base $$\big\{\{x\}:x\in X\big\}\cup\big\{\{0\}\cup U_\alpha:\alpha<|X|\big\}.$$
Claim~\ref{cl5} implies that $(X^0,\tau)$ is a Hausdorff topological semigroup contaning $X$ as a nonclosed subsemigroup, witnessing that the semigroup $X$ is not zero-closed. This is a contradiction showing that $\cov(\A_X)<|X|=\max\{2,|X|\}$.

\section{Preservation of the polyboundedness of inverse spectra}\label{s:lim}

In this section we prove a theorem on preservation of polyboundedness by inverse spectra of topological semigroups. First we recall the necessary definitions, see \cite{Chig} for more information on inverse spectra.

A {\em directed set} is a set $D$ endowed with a reflexive transitive relation $\preceq$ such that for any $x,y\in D$ there exists $z\in D$ such that $x\preceq z$ and $y\le z$. A directed set $D$ is called {\em $\sigma$-directed} if for any countable subset $C\subseteq D$ there exists $y\in D$ such that $x\preceq y$ for all $x\in C$.

For an element $\alpha$ of a directed set $D$ let ${\uparrow}\alpha\defeq\{x\in P:\alpha\preceq x\}$ be the {\em upper set} of $\alpha$ in $D$.

An {\em inverse spectrum of topological semigroups} is a family $\Sigma=(X_\alpha,\pi^\beta_\alpha:\alpha\preceq\beta\in D)$ consisting of topological semigroups $X_\alpha$ indexed by elements of a directed set $D$ and continuous homomorphisms $\pi^\beta_\alpha:X_\beta\to X_\alpha$, defined for any elements $\alpha\preceq\beta$ in $D$, such that for every elements $\alpha\preceq\beta\preceq\gamma$ in $D$ the following conditions are satisfied:
\begin{itemize}
\item $\pi^\alpha_\alpha:X_\alpha\to X_\alpha$ is the identity homomorphism of $X_\alpha$, and
\item $\pi^\gamma_\alpha=\pi^\beta_\alpha\circ\pi^\gamma_\beta$.
\end{itemize}
An inverse spectrum $\Sigma=(X_\alpha,\pi^\beta_\alpha:\alpha\preceq\beta\in D)$ is called {\em $\sigma$-directed} if its indexing directed set $D$ is $\sigma$-directed.

For an inverse spectrum of topological semigroups $\Sigma=(X_\alpha,\pi^\beta_\alpha:\alpha\preceq\beta\in D)$ its {\em limit} is the topological subsemigroup
$$\textstyle\lim\Sigma=\{(x_\alpha)_{\alpha\in D}\in\prod_{\alpha\in D}X_\alpha:\forall \alpha,\beta\in P\;(\alpha\preceq\beta\;\Rightarrow\;\pi^\beta_\alpha(x_\beta)=x_\alpha)\}$$
of the Tychonoff product $\prod_{\alpha\in D}X_\alpha$.

\begin{theorem}\label{t:lim} A power-Lindel\"of topological semigroup $X$ is polybounded if $X$ is topologically isomorphic to the limit $\lim\Sigma$ of a $\sigma$-directed inverse spectrum $\Sigma=(X_\alpha,\pi^\beta_\alpha:\alpha\preceq\beta\in D)$ consisting of polybounded Hausdorff topological semigroups.
\end{theorem}

\begin{proof} We shall identify $X$ with the limit $\lim\Sigma$ of the inverse spectrum $\Sigma$. For every $\alpha\in D$, let $$\pi_\alpha:\lim \Sigma\to X_\alpha,\quad\pi_\alpha:(x_\delta)_{\delta\in D}\mapsto x_\alpha,$$be the coordinate projection. Let $\bar \pi_\alpha:(\lim \Sigma)^1\to X_\alpha^1$ be the unique extension of $\pi_\alpha$ such that $\bar\pi_\alpha(1)=1$.

For every $\alpha\in D$, let $P_\alpha$ be the set of all semigroup polynomials on the semigroup $X_\alpha$. Since $X_\alpha$ is polybounded, there exist a number $n_\alpha\in\IN$ and sequences $(b_{\alpha,i})_{i\in n_\alpha}\in X_\alpha^{n_\alpha}$ and $(f_{\alpha,i})_{i\in n_\alpha}\in P_\alpha^{n_\alpha}$ such that $X_\alpha=\bigcup_{i\in n_\alpha}f_{\alpha,i}^{-1}(b_{\alpha,i})$.
For every $i\in n_\alpha$ choose a point $\hat b_{\alpha,i}\in q_\alpha^{-1}(b_{\alpha,i})$.

For every $i\in n_\alpha$ there exist a number $m_{\alpha,i}\in\IN$ and a sequence $(a_{\alpha,i,j})_{j=0}^{m_{\alpha,i}}\in (X_\alpha^1)^{m_{\alpha,i}+1}$ such that $$f_{\alpha,i}(x)=
a_{\alpha,i,0}xa_{\alpha,i,1}\dots xa_{\alpha,i,m_{\alpha,i}}\quad\mbox{for every $x\in X_\alpha$.}$$ For every $j\in \{0,\dots,m_{\alpha,i}\}$, choose a point $\hat a_{\alpha,i,j}\in \pi_\alpha^{-1}(a_{\alpha,i,j})$.
Consider a semigroup polynomial $\hat f_{\alpha,i}:X\to X$, defined by
$$\hat f_{\alpha,i}(x)=\hat a_{\alpha,i,0}x\hat a_{\alpha,i,1}x\cdots x\hat a_{\alpha,i,m_{\alpha,i}}\quad\mbox{for $x\in X$},$$ and observe that $$f_{\alpha,i}\circ \pi_\alpha=\pi_\alpha\circ \hat f_{\alpha,i}.$$

\begin{claim}\label{cl:finite} There exist a sequence $(m(i))_{i\in n}\in\bigcup_{k\in\IN}\IN^k$ such that for every  $\alpha\in D$ there exists $\beta\in{\uparrow}\alpha$ such that $(m_{\beta,i})_{i\in n_\beta}=(m(i))_{i\in n}$.
\end{claim}

\begin{proof} Let $\IN^{<\IN}\defeq\bigcup_{k\in\IN}\IN^k$. Assuming that Claim~\ref{cl:finite} is not true, for every sequence $\vec m\in\IN^{<\IN}$, we can find $\alpha_{\vec m}\in D$ such that $(m_{\alpha,i})_{i\in n_\alpha}\ne \vec m$ for every $\alpha\in{\uparrow}\alpha_{\vec m}$. Since the index set $D$ is $\sigma$-directed, there exists an element $\alpha\in\bigcap_{\vec m\in\IN^{<\IN}}{\uparrow}\alpha_k$. Then for the sequence $\vec m=(m_{\alpha,i})_{i\in n_\alpha}$ we get a contradiction with the choice of $\alpha_{\vec m}$.
\end{proof}

By Claim~\ref{cl:finite}, there exist a number $n\in\IN$ and a sequence $(m_i)_{i\in n}\in\IN^n$ such that the set $$A\defeq\{\alpha\in D:(m_{\alpha,i})_{i\in n_\alpha}=(m_i)_{i\in n}\}$$is cofinal in $D$ in the sense that for every $\alpha\in D$ there exists $\beta\in A\cap{\uparrow}\alpha$.

 We shall identify the product $X^n\times\prod_{i\in n}X^{m_i+1}$ with the power $X^\ell$ for $\ell\defeq n+\sum_{i\in n}(m_i+1)$.

For every $\alpha\in\A$ consider the sequence
$$s_\alpha\defeq\big((\hat b_{\alpha,i})_{i\in n},((\hat a_{\alpha,i,j})_{j=0}^{m_{i}})_{i\in n}\big)\in X^{n}\times \prod_{i\in n}X^{m_{i}+1}=X^\ell.$$

\begin{claim}\label{cl:s} There exists a sequence $s=((b_i)_{i\in n},((a_{i,j})_{j=0}^{m_i})_{i\in n})\in X^n\times\prod_{i\in n}X^{m_i+1}$ such that for every neighborhood $W$ of $s$ in $X^\ell$ and every $\delta\in D$ there exists $\alpha\in A\cap{\uparrow}\delta$ such that $s_\alpha\in W$.
\end{claim}

\begin{proof} To derive a contradiction, assume that for every $s\in X^\ell$ there exists an open neighborhood $W_s\subseteq X^\ell$ of $s$ and an element $\alpha_s\in D$ such that $s_\alpha\notin W_s$ for every $\alpha\in A\cap {\uparrow}\alpha_s$. Since the space $X$ is power-Lindel\"off, the space $X^\ell$ is Lindel\"of and hence there exists a countable set $S\subseteq X^\ell$ such that $X^\ell=\bigcup_{s\in S}W_s$.

Since the index set $D$ is $\sigma$-directed, there exists an index $\alpha\in\bigcap_{s\in S}{\uparrow}\alpha_s$. Since the set $A$ is cofinal in $D$, we can additionally assume that $\alpha\in A$. Since $s_\alpha\in X^\ell=\bigcup_{s\in S}W_s$, there exists $s\in S$ such that $s_\alpha\in W_s$. Since $\alpha\in A\cap{\uparrow}\alpha_s$, the inclusion $s_\alpha\in W_s$ contradicts the choice of the index $\alpha_s$. This contradiction completes the proof of Claim~\ref{cl:s}.
\end{proof}

Let $s=((b_i)_{i\in n},((a_{i,j})_{j=0}^{m_i})_{i\in n})$ be the sequence from Claim~\ref{cl:s}.
For every $i\in n$ consider the semigroup polynomial $f_i:X\to X$ defined by $f_i(x)=a_{i,0}xa_{i,1}\cdots xa_{i,m_i}$ for $x\in X$.

\begin{claim}\label{cl:polyb} $X=\bigcup_{i\in n}f_i^{-1}(b_i)$.
\end{claim}

\begin{proof} Assuming that the claim does not hold, we can find $x\in X$ such that $f_i(x)\ne b_i$ for all $i\in n$. The space $X=\lim\Sigma$ is Hausdorff, being a subspace of the Tychonoff product of Hausdorff topological spaces $X_\alpha$. Then for every $i\in n$ there exist open sets $V_i$ and $W_i$ in $X$ such that $b_i\in V_i$, $f_i(x)\in W_i$, and $V_i\cap W_i=\emptyset$. By the continuity of the semigroup operation on the topological semigroup $X$, there exist open sets  $W_{i,0},W_{i,1},\dots,W_{i,m_i}$ in $X$ such that $(a_{i,0},\dots,a_{i,m_i})\in \prod_{j=0}^{m_i}W_{i,j}$ and $W_{i,0}xW_{i,1}x\cdots xW_{i,m_i}\subseteq W_i$.

By the definition of the Tychonoff product topology on $X=\lim\Sigma$, there exists an index $\alpha'\in D$ and a family of open sets $\big((V'_i)_{i\in n},((W'_{i,j})_{j=0}^{m_i})_{i\in n}\big)$ in $X_{\alpha'}$ such that $$\textstyle s\in W'\defeq \prod_{i\in n} \pi_{\alpha'}^{-1}[V'_i]\times\prod_{i\in n}\prod_{j=0}^{m_i}\pi_{\alpha'}^{-1}[W'_{i,j}]\subseteq\prod_{i\in n}V_i\times\prod_{i\in n}\prod_{j=0}^{m_i}W_{i,j}.$$ By the choice of $s$, there exists $\alpha\in A\cap{\uparrow}\alpha'$ such that $s_\alpha\in W'$.
Since $X_\alpha=\bigcup_{i\in n}f_{\alpha,i}^{-1}(b_{\alpha,i})$,  there exists $i\in n$ such that $\pi_\alpha(x)\in f^{-1}_{\alpha,i}(b_{\alpha,i})$.

It follows from $s_\alpha\in W'$ that $\hat b_{\alpha,i}\in \pi^{-1}_{\alpha'}[V_i']$ and $\hat a_{\alpha,i,j}\in \pi^{-1}_{\alpha'}[W'_{i,j}]\subseteq W_{i,j}$ for all $j\le m_i$. Then
$$\hat f_{\alpha,i}(x)=\hat a_{\alpha,i,0}x\hat a_{\alpha,i,1}x\cdots x\hat a_{\alpha,i,m_i}\in W_{i,0}xW_{i,1}x\cdots xW_{i,m_i}\subseteq W_i\subseteq X\setminus V_i\subseteq X\setminus \pi_{\alpha'}^{-1}[V_i']$$ and hence $$\pi_{\alpha'}(\hat b_{\alpha,i})=\pi^\alpha_{\alpha'}(b_{\alpha,i})=\pi^\alpha_{\alpha'}(f_{\alpha,i}\circ \pi_\alpha(x))=\pi^\alpha_{\alpha'}\circ \pi_\alpha\circ \hat f_{\alpha,i}(x)=\pi_{\alpha'}\circ\hat f_{\alpha,i}(x)\notin V'_i.$$ On the other hand,
$\hat b_{\alpha,i}\in \pi^{-1}_{\alpha'}[V_i']$ implies that
 $\pi_{\alpha'}(\hat b_{\alpha,i})\in V_i'$. This contradiction completes the proof of Claim~\ref{cl:polyb}.
\end{proof}

Claim~\ref{cl:polyb} witnesses that the semigroup $X$ is polybounded.
\end{proof}

\section{Proof of Theorem~\ref{t:MA}}\label{s:MA}

The following lemma proves Theorem~\ref{t:MA}(1).

\begin{lemma}\label{l:MA1} Let $X$ be a zero-closed semigroup admitting a separable complete subinvariant metric $d$. If $\cov(\M)=\mathfrak c$, then the semigroup $X$ is polybounded.
\end{lemma}

\begin{proof} Let $\tau$ be the topology on $X$, generated by the subinvariant metric $d$. By Proposition~\ref{p:ideal}, the ideal $I=\bigcap_{a\in X}XaX$ is not empty. If $I$ contains an isolated point $o$, then there exists $\e>0$ such that the set $\{x\in I:d(x,o)<\e\}=\{o\}$ is polybounded. By Theorem~\ref{t:local}, the semigroup $X$ is polybounded.

So, now assume that the minimal ideal $I$ contains no isolated points. Then its closure $\overline I$ in $X$ also contains no isolated points. The Lavrentiev Theorem \cite[3.9]{Ke} implies that the Polish space $\overline I$ contains a  dense $G_\delta$-subset, homeomorphic to the space $\mathbb \IR\setminus\mathbb Q$ of rationals. Now the definition of the cardinal $\cov(\M)=\mathfrak c$ implies that the Polish space $\overline I$ cannot be covered by less than continuum many nowhere dense sets. By Theorem~\ref{t:poly}, $\kappa\defeq \cov(\A_X)<|X|=\mathfrak c$. Then $X=\bigcup_{\alpha\in\kappa}f_\alpha^{-1}(b_\alpha)$ for some elements $b_\alpha\in X$ and semigroup polynomials $f_\alpha$ on $X$. The subinvariance of the metric $d$ implies that the binary operation of the semigroup $X$ is continuous with respect to the topology $\tau$. Then for every $\alpha\in\kappa$ the semigroup polynomial $f_\alpha:X\to X$ is continuous and the set $f_\alpha^{-1}(b_\alpha)$ is closed in $(X,\tau)$. Since $\overline I=\bigcup_{\alpha\in\kappa}(f_\alpha^{-1}(b_\alpha)\cap\overline I)$ and
 $\kappa<\mathfrak c=\cov(\M)$, one of the sets $\overline I\cap f_{\alpha}^{-1}(b_\alpha)$ has non-empty interior in $\overline I$. So, we can find a point $o\in I\cap f^{-1}_\alpha(b_\alpha)$ and $\e>0$ such that the set  $\{x\in I:d(x,o)<\e\}$ is contained in $f_\alpha^{-1}(b_\alpha)$ and hence is polybounded in $X$. By Theorem~\ref{t:local}, the semigroup $X$ is polybounded.
\end{proof}

In the proof of Theorems~\ref{t:MA}(2) and \ref{t:pf} we shall need the following lemma.

\begin{lemma}\label{l:zp} Let $h:X\to Y$ be a  surjective homomorphism. If the semigroup $X$ is zero-closed, then the semigroup $Y$ is zero-closed, too.
\end{lemma}

\begin{proof} Assuming that $Y$ is not zero-closed, we conclude that $Y$ is not closed in its $0$-extension $Y^0=\{0\}\cup Y$, endowed with some Hausdorff semigroup topology $\tau^0_Y$. Consider the topology $\tau_X^0$ on $X^0$, generated by the base
$$\big\{\{x\}:x\in X\big\}\cup\big\{\{0\}\cup h^{-1}(U):0\in U\in\tau^0_Y\big\},$$
and observe that $\tau_X^0$ is a Hausdorff semigroup topology on $X^0$ and $X$ is not closed in $(X^0,\tau_X^0)$, which is not possible as $X$ is zero-closed. This contradiction shows that the semigroup $Y$ is zero-closed.
\end{proof}

The following Lemma proves Theorem~\ref{t:MA}(2).

\begin{lemma}\label{l:MA2} Assume that $\cov(\M)=\mathfrak c$. Each zero-closed compact Hausdorff topological semigroup $X$ is polybounded.
\end{lemma}

\begin{proof} Let $D$ be the set of continuous subinvariant pseudometrics on the compact topological semigroup $X$. Each pseudometric $\alpha\in D$ determines the congruence $\alpha^{-1}(0)$ on the semigroup $X$. Let $X_\alpha$ be the quotient semigroup of $X$ by the congruence $\alpha^{-1}(0)$, and $\pi_\alpha:X\to X_\alpha$ be the quotient homomorphism. The pseudometric $\alpha$ induces a metric $\check\alpha$ on $X_\alpha$ such that $\check\alpha(\pi_\alpha(x),\pi_\alpha(y))=\alpha(x,y)$ for any $x,y\in X$. The continuity of the pseudometric $\alpha$ and the compactness of $X$ implies that the topology on $X_\alpha$ generated by the metric $\check\alpha$ coincides with the quotinent topology induced by the quotient homomorphism $\pi_\alpha:X\to X_\alpha$. The compactness of $X_\alpha$ ensures that the subinvariant metric $\check\alpha$ on $X_\alpha$ is separable and complete.
The subinvariance of the meric $\check\alpha$ implies that $X_\alpha$ is a topological semigroup. Talking about meric properties of the semigroup $X_\alpha$ we always will have in mind the metric $\check \alpha$ on $X_\alpha$.

By Lemma~\ref{l:zp}, the semigroup $X_\alpha$ is zero-closed and by Lemma~\ref{l:MA1}, $X_\alpha$ is polybounded.

Given two pseudometrics $\alpha,\beta\in D$ we write $\alpha\preceq\beta$ if there exists a positive real number $C$ such that $\alpha(x,y)\le C\cdot\beta(x,y)$ for all $x,y\in X$. This inequality implies the inclusion $\beta^{-1}(0)\subseteq\alpha^{-1}(0)$ which allows us to define a (unique) homomorphism $\pi^\beta_\alpha:X_\beta\to X_\alpha$ such that $\pi_\alpha=\pi^\beta_\alpha\circ\pi_\beta$. The definition of the relation $\alpha\preceq\beta$ implies that the homomorphism $\pi^\beta_\alpha:X_\beta\to X_\alpha$ is Lipschitz and hence continuous.

Let us show that the poset $D$ is $\sigma$-directed. Indeed, given any countable set $\{\alpha_n\}_{n\in\w}\subseteq D$, for every $n\in\w$ choose a positive real number $c_n$ such that $c_n\cdot \sup_{x,y\in X}\alpha_n(x,y)\le\frac1{2^n}$. Then the pseudometric $\alpha=\max_{n\in\w}c_n\alpha_n$ on $X$ is continuous and subinvariant, and hence belongs the directed set $D$. It is clear that $\alpha_n\preceq\alpha$ for every $n\in\w$. This means the directed set $D$ is $\sigma$-directed.

Now we see that $\Sigma=(X_\alpha,\pi_\alpha^\beta:\alpha\preceq\beta\in D)$ is a $\sigma$-directed spectrum of compact metrizable topological semigroups. Then its limit $\lim\Sigma$ is a compact Hausdorff space, being a closed subspace of the Tychonoff product $\prod_{\alpha\in D}X_\alpha$ of compact metrizable spaces. For every $\alpha\in D$, let $\pr_\alpha:\lim\Sigma\mapsto X_\alpha$ be the $\alpha$-th coordinate projection.

Since each semigroup $X_\alpha$, $\alpha\in D$, is polybounded, the compact topological semigroup $\lim \Sigma$ is polybounded by Theorem~\ref{t:lim}.

It remains to show that the semigroup $X$ is isomorphic to $\lim\Sigma$. The quotient homomorphisms $\pi_\alpha:X\to X_\alpha$, $\alpha\in D$, determine a continuous homomorphism $\pi:X\to\lim\Sigma$, $\pi:x\mapsto(\pi_\alpha(x))_{\alpha\in D}$, such that $\pi_\alpha=\pr_\alpha\circ\pi$ for every $\alpha\in D$. To show that $\pi$ is injective, take any distinct points $x,y\in X$ and choose any continuous function $f:X\to \IR$ such that $f(x)\ne f(y)$. By the compactness of $X$ and the continuity of the semigroup operation on $X$, the formula
$$\alpha(x,y)=\max_{a,b\in X^1}|f(axb)-f(ayb)|$$
determines a continuous subinvariant pseudometric on $X$ such that $\alpha(x,y)\ge |f(x)-f(y)|>0$. Then for the quotient homomorphism $\pi_\alpha:X\to X_\alpha$ we get $\pi_\alpha(x)\ne\pi_\alpha(y)$ and hence $\pi(x)\ne \pi(y)$.  Therefore, the map $\pi:X\to \lim\Sigma$ is injective.

Assuming that $\pi$ is not surjective, we can use the closedness of the compact  set $\pi[X]$ in $\lim\Sigma\subseteq\prod_{\alpha\in D}X_\alpha$ and find  $\alpha\in D$ and an open set $V\subseteq X_\alpha$ such that $\pi[X]\cap \pi^{-1}_\alpha[V]\ne\emptyset$. Then $V\ne\emptyset$ and by the surjectivity of the quotient homomorphism $\pi_\alpha:X\to X_\alpha$, there exists $x\in X$ such that $\pi_\alpha(x)\in V$ and hence $\pi(x)\in \pr_\alpha^{-1}(\pi_\alpha(x))\in \pr_\alpha^{-1}[V]$, which contradicts $\pi[X]\cap\pr_\alpha^{-1}[V]=\emptyset$. This contradiction shows that $\pi$ is surjective and hence bijective. Therefore, $\pi:X\to\lim\Sigma$ is an isomorphism and the polyboundedness of the semigroup $\lim\Sigma$ implies the polyboundedness of $X$.

\end{proof}

\section{Polyfinite semigroups}\label{s:polyfinite}

We recall that a semigroup $X$ is {\em polyfinite} if there exist $n\in\IN$ and a finite set $F\subseteq X$ such that for every $x,y\in X$ there exists a semigroup polynomial $f:X\to X$ of degree $\le n$ such that $\{f(x),f(y)\}\subseteq F$.

\begin{lemma}\label{l:pb=>pf} Each polybounded semigroup is polyfinite.
\end{lemma}

\begin{proof}  Since $X$ is polybounded, there exist elements $b_0,\dots,b_{n-1}\in X$ and semigroup polynomials $f_0,\dots,f_{n-1}$ on $X$ such that $X=\bigcup_{i\in n}f_i^{-1}(b_i)$. Let $$F\defeq\{b_i\}_{i\in n}\cup\{f_i(b_j):i,j\in n\}\quad\mbox{and}\quad d\defeq\max\{\deg(f_i\circ f_j):i,j\in n\}.$$ Given any elements $x,y\in X$, find $i\in n$ such that $f_i(x)=b_i$ and then find $j\in n$ such that $f_j(f_i(y))=b_j$. The semigroup polynomial $f=f_j\circ f_i:X\to X$ has degree $\le d$ and $\{f(x),f(y)\}=\{f_j(f_i(x)),f_j(f_i(y))\}=\{f_j(b_i),b_j\}\subseteq F$, witnessing that $X$ is polyfinite.
\end{proof}

A permutation $\pi:X\to X$ of a set $X$ is called {\em finitely supported} if it has finite {\em support} $\sup(\pi)=\{x\in X:\pi(x)\ne x\}$. It is clear that the set $FS_X$ of all finitely supported permutations of the set $X$ is a group with respect of the operation of composition of permutations.

\begin{example}\label{ex:AR} For any infinite set $X$, the group $FS_X$ of all finitely supported permutations of $X$ is polyfinite but not polybounded.
\end{example}

\begin{proof} First, we show that $FS_X$ is  polyfinite. Given any finitely supported permutations  $\pi_1,\pi_2\in FS_X$, consider a semigroup polynomial $$f:FS_X\to FS_X,\quad f:x\mapsto \pi_1^{-1}yc\pi_1^{-1}yc^{-1},$$ where $c \in FS_X$ is a permutation such that $c\pi_1^{-1}\pi_2 c^{-1}=(\pi_1^{-1}\pi_2)^{-1}$. In this case $$f(\pi_1)=\pi_1^{-1}\pi_1 c\pi_1^{-1}\pi_1 c^{-1} = 1\mbox{ and }f(\pi_2)=\pi_1^{-1}\pi_2 c\pi_1^{-1}\pi_2 c^{-1}=\pi_1^{-1}\pi_2 (\pi_1^{-1}\pi_2)^{-1}=1,$$ witnessing that the group $FS_X$ is polyfinite.
\smallskip

Next, we show that the group $FS_X$ is not polybounded. Indeed, in the opposite case, $FS_{X}=\bigcup_{i\in n}f_i^{-1}(b_i)$ for some elements $b_0,\dots,b_{n-1}\in FS_X$ and semigroup polynomials $f_0,\dots,f_{n-1}$ on $FS_X$. For every $i\in n$ find a number $m_i\in \IN$ and elements $c_{i,0},\cdots c_{i,m_i}\in FS_X$ such that $f_i(y)=c_{i,0}yc_{i,1}y\cdots yc_{i,m_i}$ for all $y\in FS_X$. Consider the finite set $S=\bigcup_{i\in n}\big(\supp(b_i)\cup\bigcup_{j=0}^{m_{i}}\supp(c_{i,j})\big)$ and the number $m=\max_{i\in n}m_i$. Choose any sequence $(x_0,\dots, x_{m})$ of pairwise distinct points of the set $X\setminus S$. Let $\pi:X\to X$ be the permutation of $X$ defined by
$$\pi(x)=\begin{cases}x_{i+1}&\mbox{if $x=x_i$ for some $i<m$};\\
x_0&\mbox{if $x=x_m$};\\
x&\mbox{otherwise}.
\end{cases}
$$  Since $\pi\in FS_X\in \bigcup_{i\in n}f^{-1}(b_i)$, there exists $i\in n$ such that $f_i(\pi)=b_i$. 
It follows from $x_0\notin \supp(b_i)$ that $b_i(x_0)=x_0$. On the other hand, $\{x_0,\dots,x_{m}\}\cap \bigcup_{j=0}^{m_i}\supp(c_{i,j})=\emptyset$ implies $f_i(\pi)(x_0)=\pi^{m_i}(x_0)=x_{m_i}\ne x_0=b_i(x_0)$ and hence $f_i(\pi)\ne b_i$. This contradiction completes the proof.
\end{proof}

\begin{theorem}\label{t:lim-polyfinite} A compact topological semigroup $X$ is polyfinite if $X$ is topologically isomorphic to the limit $\lim\Sigma$ of a $\sigma$-directed inverse spectrum $\Sigma=(X_\alpha,\pi^\beta_\alpha:\alpha\preceq\beta\in D)$ consisting of polyfinite Hausdorff topological semigroups.
\end{theorem}

\begin{proof} We shall identify $X$ with the limit $\lim\Sigma$ of the inverse spectrum $\Sigma$. For every $\alpha\in D$ let $$\pi_\alpha:\lim \Sigma\to X_\alpha,\quad\pi_\alpha:(x_\delta)_{\delta\in D}\mapsto x_\alpha,$$be the coordinate projection. Let $\bar \pi_\alpha:(\lim \Sigma)^1\to X_\alpha^1$ be the unique extension of $\pi_\alpha$ such that $\bar\pi_\alpha(1)=1$.

Since $X_\alpha$ is polyfinite, there exist a number $n_\alpha\in\IN$ and a finite set $F_\alpha\subseteq X_\alpha$ such that for any $x,y\in X_\alpha$ there exists a semigroup polynomial $f_\alpha:X_\alpha\to X_\alpha$ of degree $\le n_\alpha$ such that $\{f_\alpha(x),f_\alpha(y)\}\subseteq F_\alpha$.
Let $m_\alpha=|F_\alpha|$ and choose elements $b_{\alpha,1},\dots,b_{\alpha,m_\alpha}\in X_\alpha$ such that $\{b_{\alpha,1},\dots,b_{\alpha,m_\alpha}\}=F_\alpha$. For every $i\in\{1,\dots,m_\alpha\}$ choose an element $\hat b_{\alpha,i}\in X$ such that $\pi_\alpha(\hat b_{\alpha,i})=b_{\alpha,i}$.

Repeating the argument from Claim~\ref{cl:finite}, we can find numbers $n,m\in\IN$ such that the set $$A\defeq\{\alpha\in D:(n_\alpha,m_\alpha)=(n,m)\}$$ is cofinal in $D$.

A family $\U$ of nonempty subsets of $D$ is called a {\em filter on $D$} if
\begin{itemize}
\item for any $U,V\in\U$ we have $U\cap V\in\U$;
\item for any sets $V\subseteq W\subseteq D$, the inclusion $V\in \U$ implies $W\in\U$.
\end{itemize}
A filter $\U$ on $D$ is called an {\em ultrafilter} if $\U=\V$ for every filter $\V$ on $D$ with $\U\subseteq \V$.
The Kuratowski--Zorn Lemma implies that every filter on $D$ can be enlarged to an ultrafilter. In particular, the filter $\{F\subseteq D:\exists \alpha\in D\;\;A\cap{\uparrow}\alpha\subseteq F\}$ can be enlarged to some ultrafilter $\U$ on $D$.

\begin{claim}\label{cl2:s} There exists a sequence $b=(b_1,\dots,b_m)\in X^m$ such that for every neighborhood $W$ of $b$ in $X^m$ the set $\{\alpha\in A:(\hat b_{\alpha,1},\dots,\hat b_{\alpha,m})\in W\}$ belongs to the ultrafilter $\U$.
\end{claim}

\begin{proof} To derive a contradiction, assume that for every $b\in X^m$ there exists an open neighborhood $W_b\subseteq X^m$ of $b$ such that the set $A_b=\{\alpha\in A:(\hat b_{\alpha,1},\dots,\hat b_{\alpha,m})\in W_b\}$ does not belong to the ultrafilter $\U$. By the maximality of $\U$, there exists a set $U_b\in\U$ such that the intersection $A_b\cap U_b$ is empty (in the opposite case we could enlarge the filter $\U$ to a strictly larger filter $\{W\subseteq D:\exists U\in\U\;\;(U\cap A_b\subseteq W)\}$).
By the compactness of the space $X^m$, there exists a finite set $B\subseteq X^m$ such that $X^m\subseteq \bigcup_{b\in B}W_b$. Consider the (nonempty) set $U_B=\bigcap_{b\in B}U_b\in\U$ and choose any element $\alpha\in U_B\cap A$. Since $(\hat b_{\alpha,1},\dots,\hat b_{\alpha,m})\in X^m=\bigcup_{b\in B}W_b$, there exists $b\in B$ such that $(\hat b_{\alpha,1},\dots,\hat b_{\alpha,m})\in W_b$. Then $\alpha\in A_b\cap U_b=\emptyset$, which is a contradiction completing the proof of the claim.
\end{proof}

Consider the finite set $F=\{b_1,\dots,b_m\}$ in $X$. We claim that for every points $\hat x,\hat y\in X$ there exists a semigroup polynomial $f:X\to X$ such that $\{f(\hat x),f(\hat y)\}\subseteq F$.

Given any $\hat x,\hat y\in X$, for every $\alpha\in D$ find a semigroup polynomial $f_\alpha:X_\alpha\to X_\alpha$ of degree $d_\alpha\le n_\alpha$ such that $\{f_\alpha(\pi_\alpha(\hat x)),f_\alpha(\pi_\alpha(\hat y))\}\subseteq F_\alpha$. Next, find elements $a_{\alpha,0},\dots,a_{\alpha,d_\alpha}\in X_\alpha$ such that $f_\alpha(x)=a_{\alpha,0}xa_{\alpha,1}x\cdots xa_{\alpha,d_\alpha}$ for every $x\in X_\alpha$. For every $i\in\{0,\dots,d_\alpha\}$, choose an element $\hat a_{\alpha,i}\in \pi_\alpha^{-1}(a_{\alpha,i})$, and consider the semigroup polynomial
$$\hat f:X\to X,\quad \hat f:x\mapsto \hat a_{\alpha,0}x\hat a_{\alpha,1}x\cdots x\hat a_{\alpha,m}.$$

For every $d\le n$ consider the set $A_d=\{\alpha\in A:d_\alpha=d\}$. Since $A=\bigcup_{d=1}^nA_d\in\U$, there exists $d\in\{1,\dots,n\}$ such that $A_d\in\U$.

Repeating the argument from Claim~\ref{cl2:s}, we can find a sequence $(a_0,\dots,a_d)\in X^{d+1}$ such that for every neighborhood $W$ of $(a_0,\dots,a_d)$ in  $X^{d+1}$, the set $\{\alpha\in D:(\hat a_{\alpha,0},\dots,\hat a_{\alpha,d})\in W\}$ belongs to the ultrafilter $\U$.

Consider the semigroup polynomial $f:X\to X$, $f:x\mapsto a_0xa_1x\cdots xa_d$.
We claim that $\{f(\hat x),f(\hat y)\}\subseteq F$. Assuming that $f(\hat x)\notin F$, we can find open sets $W$ and $V_1,\dots,V_m$ in $X$ such that $W\cap\bigcup_{i=1}^mV_i=\emptyset$, $f(\hat x)\in W$, and $b_i\in V_i$ for every $i\in\{1,\dots, m\}$. By the continuity of the semigroup operation on the topological semigroup $X$, there exist open sets  $W_{0},W_{1},\dots,W_{d}$ in $X$ such that $(a_{0},\dots,a_{d})\in \prod_{i=0}^{d}W_i$ and $W_0\hat xW_1\hat x\cdots \hat xW_d\subseteq W$.

By the definition of the Tychonoff product topology on $X=\lim\Sigma$, there exists an index $\alpha'\in D$ and a sequence of open sets $(V'_i)_{i=1}^m$ in $X_{\alpha'}$ such that $$(b_1,\dots,b_m)\textstyle \in V'\defeq \prod_{i=1}^m\pi_{\alpha'}^{-1}[V'_i]\subseteq\prod_{i=1}^m V_i.$$ By the choice of the sequences $(b_1,\dots,b_m)$ and $(a_0,\dots,a_d)$, the set $$\textstyle A'\defeq\bigcap_{i=1}^m\{\alpha\in A_d:\hat b_{\alpha,i}\in \pi_\alpha^{-1}[V_i]\}\cap\bigcap_{i=0}^d\{\alpha\in A_d:\hat a_{\alpha,i}\in W_i\}$$ belongs to the ultrafilter $\U$. Since $\emptyset \ne A'\cap {\uparrow}\alpha\in\U$ for every $\alpha\in D$, the set $A'$ is cofinal in $D$ and hence  there exists an element $\alpha\in A'$ such that $\alpha'\preceq\alpha$. Since $f_\alpha(\pi_\alpha(\hat x))\in F_\alpha$, there exits $i\in \{1,\dots,m\}$ such that $f_\alpha(\pi_\alpha(\hat x))=b_{\alpha,i}$ and hence
$\pi_\alpha(\hat f_\alpha(\hat x))=f_\alpha(\pi_\alpha(\hat x))=b_{\alpha,i}$.  Then
$$\pi_{\alpha'}(\hat f_\alpha(\hat x))=\pi^\alpha_{\alpha'}(\pi_\alpha(\hat f_\alpha(\hat x)))=\pi_{\alpha'}^\alpha(b_{\alpha,i})=\pi_{\alpha'}^\alpha(\pi_\alpha(\hat b_{\alpha,i}))=\pi_{\alpha'}(\hat b_{\alpha,i})\in V_i'$$and
$$\hat f_\alpha(\hat x)\in \pi_{\alpha'}^{-1}[V_i']\subseteq V_i\subseteq X\setminus W.$$
On the other hand,
$\hat f_\alpha(\hat x)=\hat a_{\alpha,0}\hat x\hat a_{\alpha,1}\hat x\cdots \hat x\hat a_{\alpha,d}\in W_{0}\hat xW_{1}\hat x\cdots \hat xW_{d}\subseteq W$, which is a contradiction completing the proof of the inclusion $f(\hat x)\in F$. By analogy we can prove that $f(\hat y)\in F$, witnessing that the semigroup $X$ is polyfinite.
\end{proof}

\section{Proof of Theorem~\ref{t:pf}}\label{s:pf}

We divide the proof of Theorem~\ref{t:pf} into five lemmas.

\begin{lemma}\label{l:cov-eXe} For every idempotent $e$ of a semigroup $X$ we have
$\cov(\A_X)=\cov(\A_{eXe})$.
\end{lemma}

\begin{proof} To prove that $\cov(\A_X)\le\cov(\A_{eXe})$, fix a family $\C\subseteq\A_{eXe}$ of cardinality $|\C|=\cov(\A_{eXe})$ such that $eXe=\bigcup\C$. For every $C\in\C$, find a number $n_C$ and elements $b_C\in eXe$ and $a_{C,0},\dots,a_{C,n_C}\in (eXe)^1$ such that $C=\{x\in eXe:a_{C,0}xa_{C,1}x\cdots xa_{C,n_C}=b_C\}$. Replacing each coefficient $a_{C,i}$ by $ea_{C,i}e$, we can assume that $a_{C,i}=ea_{C,i}e\in eXe$ for all $i\in\{0,\dots,n_C\}$.
Let $\overline C\defeq\{x\in X:a_{C,0}xa_{C,1}x\cdots xa_{C,n_C}=b_C\}\in\A_X$.

We claim that $X=\bigcup_{C\in\C}\overline C$. Indeed, for any $x\in X$, there exists $C\in\C$ such that $exe\in C$ and then
$$a_{C,0}xa_{C,1}x\cdots xa_{C,n_C}=a_{C,0}exea_{C,1}exe\dots exea_{C,n_C}=b_C$$ and hence $x\in\overline C$. Therefore, $\cov(\A_X)\le|\{\overline C:C\in\C\}|\le|\C|=\cov(\A_{eXe})$.
\smallskip

To see that $\cov(\A_{eXe})\le\cov(\A_X)$, fix any cover $\C\subseteq\A_X$ of $X$ with $|\C|=\cov(\A_X)$. For every  $C\in\C$, find a number $n_C$ and elements $b_C\in X$ and $a_{C,0},\dots,a_{C,n_C}\in X^1$ such that $C=\{x\in X:a_{C,0}xa_{C,1}x\cdots xa_{C,n_C}=b_C\}$.  Let $b'_C=eb_Ce\in eXe$ and for every $i\in\{0,\dots,n_C\}$, let $a_{C,i}'= ea_{C,i}e\in eXe$. Then the set $$C'\defeq\{y\in eXe:a'_{C,0}ya'_{C,1}y\cdots ya'_{C,n_C}=b'_C\}$$ belongs to the family $\A_{eXe}$.

We claim that $eXe=\bigcup_{C\in\C}C'$. Given any $x\in eXe$, find $C\in \C$ such that $exe\in C$. Then
$$a'_{C,0}xa'_{C,1}x\cdots xa'_{C,n}=ea_{C,0}exea_{C,1}ex\dots xea_{C,n_C}e=eb_Ce=b_C'$$ and hence $x\in C'$. Therefore, $eXe=\bigcup_{C\in\C}C'$ and $$\cov(\A_{eXe})\le|\{C':C\in\C\}|\le|\C|=\cov(\A_{eXe}).$$
\end{proof}

\begin{lemma}\label{l:power} Let $X$ be a group and $V$ be a subset of $X$ such that $X=VF$ for some finite set $F\subseteq X$. Then for every $x,y\in X$ there exists a positive number $k\le |F|^2$ such that $\{x^k,y^k\}\subseteq VV^{-1}$.
\end{lemma}

\begin{proof}  For every $a,b\in F$, consider the set  $N_{a,b}=\{n\in\{1,\dots,|F|^2+1\}:x^n\in Va,\;y^n\in Vb\}$. By the Pigeonhole Principle, for some $a,b\in F$ the set $N_{a,b}$ contains two numbers $i<j$. Then for the number $k=j-i\le|F|^2$ we have $x^k=x^{j-i}=x^j(x^i)^{-1}\in Va(Va)^{-1}=VV^{-1}$ and $y^k\in (Vb)(Vb)^{-1}=VV^{-1}$.
\end{proof}

\begin{lemma}\label{l:Stein} A group $X$ is polyfinite if there exist $n\in\IN$ and an algebraic set $A\in\A_X$ such that for every $x,y\in X$ there exist $c\in X$ and a positive number $k\le n$ such that $x^ky^k\in c(AA^{-1}\cup A^{-1}A)c^{-1}$.
\end{lemma}

\begin{proof} Since the set $A$ is algebraic in the group $X$, there exists a semigroup polynomial $\alpha:X\to X$ such that $A=\alpha^{-1}(1)$. Given any elements $x,y\in X$, by our assumption there exist and a positive number $k\le n$ such that $(x^{-1})^ky^k\in\{cab^{-1}c^{-1},ca^{-1}bc^{-1}\}$ for some $a,b\in A$, and $c\in C$.

If $(x^{-1})^ky^k=cab^{-1}c^{-1}$, then consider the semigroup polynomial $f:X\to X$, $f:z\mapsto \alpha(c^{-1}x^{-k}z^kcb)$, and observe that $\deg f=k\deg(\alpha)\le n\deg(\alpha)$ and
$$\{f(x),f(y)\}=\{\alpha(c^{-1}x^{-k}x^kcb),\alpha(c^{-1}x^{-k}y^kcb)\}=\{\alpha(b),\alpha(a)\}=\{1\},$$witnessing that the group $X$ is polyfinite.

If $(x^{-1})^ky^k=ca^{-1}bc^{-1}$, then consider the semigroup polynomial $f:X\to X$, $f:z\mapsto \alpha(ac^{-1}x^{-k}z^kc)$, and observe that $\deg f=k\deg(\alpha)\le n\deg(\alpha)$ and
$$\{f(x),f(y)\}=\{\alpha(ac^{-1}x^{-k}x^kc),\alpha(ac^{-1}x^{-k}y^kc)\}=\{\alpha(a),\alpha(b)\}=\{1\},$$witnessing that the group $X$ is polyfinite.
\end{proof}

A topology $\tau$ on a semigroup $X$ is called {\em monomial} if for every $a,b\in X^1$, the function $X\to X$, $x\mapsto axb$, is continuous with respect to the topology $\tau$. A topology $\tau$ on a semigroup $X$ is monomial if and only if the semigroup operation $X\times X\to X$, $(x,y)\mapsto xy$, is separately continuous.

\begin{lemma}\label{l:barN} Assume that $\cov(\overline{\mathcal N})=\mathfrak c$. A zero-closed semigroup $X$ is polyfinite if $X$ admits a compact Polish monomial  topology $\tau$.
\end{lemma}

\begin{proof}  By \cite[2.7]{HS}, $X$ contains an idempotent $e$ such that $H=eXe$ is a subgroup of $X$. By \cite[2.39]{HS}, $H=eXe$ is a compact topological group. By the classical result of Haar (see, e.g. \cite{Alfsen}), the compact topological group $H$ has a Haar measure $\lambda$, i.e., a unique probability invariant Borel measure on $H$. Let $\overline{\mathcal N}_{\!H}$ be the family of all closed subsets of $H$ of Haar measure zero, and $$\textstyle\cov(\overline{\mathcal N}_{\!H})\defeq\min\{|\C|:\C\subseteq\overline{\mathcal N}_{\!H},\;\;H=\bigcup\C\}.$$ It is well-known (see, e.g. \cite[Corollary 5.2]{BRZ}) that every infinite Polish locally compact group $G$ has $\cov(\overline{\mathcal N}_G)=\cov(\overline{\mathcal N})$. In particular,  $\cov(\overline{\mathcal N}_{\!H})=\cov(\overline{\mathcal N})=\mathfrak c$.

By \cite[1.5.1]{Eng}, the Polish space $(X,\tau)$ has cardinality $|X|\le\mathfrak c$. By Theorem~\ref{t:poly} and Lemma~\ref{l:cov-eXe}, $\cov(\A_{\!H})=\cov(\A_X)<\max\{2,|X|\}\le\mathfrak c=\cov(\overline{\mathcal N})=\cov(\overline{\mathcal N}_{\!H})$. Then there exists a set $A\in\A_{H}\setminus\overline{\mathcal N}_{H}$. By the Steinhaus--Weil Theorem \cite{Strom}, $AA^{-1}$ is a neighborhood of the identity in $H$. Since $H$ is a topological group, there exists a neighborhood $U\in\tau$ of the identity such that $U^{-1}UUU^{-1}\subseteq AA^{-1}$. By the compactness of the topological group $H$, there exists a finite set $F\subseteq H$ such that $H=UF$. By Lemma~\ref{l:power}, for every $x,y\in X$ there exists a positive number $k\le |F|^2$ such that $\{x^k,y^k\}\in UU^{-1}$ and hence $x^ky^k\in UU^{-1}UU^{-1}\subseteq AA^{-1}$. By Lemma~\ref{l:Stein}, the group $X$ is polyfinite.
\end{proof}

The following lemma implies Theorem~\ref{t:pf}, announced in the introduction.

\begin{lemma} Assume that  $\cov(\overline{\mathcal N})=\mathfrak c$. A zero-closed semigroup $X$ is polyfinite if $X$ admits a compact Hausdorff semigroup topology.
\end{lemma}

\begin{proof} This lemma can be deduced from Lemma~\ref{l:barN} and Theorem~\ref{t:lim-polyfinite} by analogy with the proof of Lemma~\ref{l:MA2}.
\end{proof}

\section{Proof of Corollary~\ref{c:IC-M}}\label{s:IC-M}

In this section we prove Corollary~\ref{c:IC-M}. First we prove a lemma linking injectively $\C$-closed and zero-closed semigroups.

\begin{lemma}\label{l1} Let $\C$ be a class of topological semigroups, containing all $0$-discrete Hausdorff topological semigroups. Every injectively $\C$-closed semigroup $X$ is zero-closed.
\end{lemma}

\begin{proof} Assuming that a semigroup $X$ is not zero-closed, we can find a Hausdorff semigroup topology $\tau$ on $X^0$ such that $X$ is not closed in the topological semigroup $(X^0,\tau)$. Let $\tau'$ be the topology on $X^0$, generated by the base
$$\tau\cup \big\{\{x\}:x\in X\big\}.$$
It is easy to see that $(X^0,\tau')$ is a $0$-discrete Hausdorff topological semigroup. Then $(X^0,\tau')\in\C$ and the injective $\C$-closedness of $X$ ensures that $X$ is closed in $(X^0,\tau')\in\C$ and hence $0$ is an isolated point in $(X^0,\tau')$ and also in $(X,\tau)$, which contradicts the choice of the topology $\tau$.
\end{proof}

The following Lemma implies Corollary~\ref{c:IC-M}.

\begin{lemma} Let $\C$ be a class of topological semigroups, containing all Polish topological semigroups and all $0$-discrete Hausdorff topological semigroups. An injectively $\C$-closed subsemigroup $X$ is polybounded if $|X|\le\cov(\M)$ and $X$ admits a separable subinvariant metric.
\end{lemma}

\begin{proof} Assume that an injectively $\C$-closed semigroup $X$ has cardinality $|X|\le\cov(\M)$ and admits a separable subinvariant metric $d$.

Let $\overline X$ be the completion of the metric space $(X,d)$. Endow the set $X\times X$ with the metric $$d_2\big((x,y),(x',y'))=d(x,x')+d(y,y')$$and observe that for every $(x,y),(x',y')\in X\times X$ we have
$$d(xy,x'y')\le d(xy,x'y)+d(x'y,x'y')\le d(x,x')+d(y,y')=d_2((x,y),(x',y'))$$ by the subinvariance of the metric $d$. This implies that the semigroup operation $\cdot: X\times X\to X$, $\cdot:(x,y)\mapsto xy$, is uniformly continuous and hence can be extended to a uniformly continuous binary operation $\bar\cdot:\overline X\times\overline X\to\overline X$, see \cite[4.2.17]{Eng}. The associativity of the semigroup operation on $X$ and the density of $X$ in $\overline X$ implies the associativity of the extended binary operation on $\overline X$. Now we see that the separable compelete metric space $\overline X$ endowed with the extended associative operation $\bar \cdot$ is a Polish topological semigroup. Since $X\subseteq \overline X\in\C$, the injective $\C$-closedness of the semigroup $X$ ensures that $X=\overline X$, which means that the separable subinvariant metric $d$ on $X$ is complete.

If $|X|=\mathfrak c$, then the inequality $|X|\le\cov(\M)$ implies $\cov(\M)=\mathfrak c$. By Lemma~\ref{l1}, the injectively $\C$-closed semigroup $X$ is zero-closed and by Theorem~\ref{t:MA}(1), the zero-closed semigroup $X=\overline X$ is polybounded.

If $|X|<\mathfrak c$, then the Polish space $\overline X=X$ is countable by \cite[6.5]{Ke}. By Lemma~\ref{l1}, the injectively $\C$-closed semigroup $X$ is zero-closed and by Theorem~\ref{t:poly}, the countable zero-closed semigroup $X$ is polybounded.
\end{proof}

\section{Proof of Corollary~\ref{c:IC-C}}\label{s:IC-C}

 Let $\C$ be a class of topological semigroups, containing all Hausdorff topological semigroups which are $0$-discrete or compact. Let $X$ be an injectively $\C$-closed subsemigroup $X$ of a compact Hausdorff topological semigroup $K$.  Since $X\subseteq K\in\C$, the injectively $\C$-closed semigroup $X$ is closed in $K$ and hence $X$ carries a compact Hausdorff semigroup topology. By Lemma~\ref{l1}, the injectively $\C$-closed semigroup $X$ is zero-closed. By Theorem~\ref{t:MA}(3) (and \ref{t:MA}(2)), the zero-closed semigroup $X$ is polyfinite (and polybounded) if  $\cov(\overline{\mathcal N})=\mathfrak c$ (and $\cov(\M)=\mathfrak c$).

\section{Proof of Corollary~\ref{c:AC}}\label{s:AC}

Let $\C$ be a class of topological semigroups containing all Tychonoff topological semigroups and $X$ be a zero-closed absolutely $\C$-closed semigroup admitting a power-Lindel\"of subinvariant Hausdorff topology $\tau$.
Endow the semigroup $X$ with the topology $\tau$. Assume that $\cov(\M)=\mathfrak c$.

Let $D$ be the family of all $\tau$-continuous subinvariant pseudometrics on the semigroup $X$. Each pseudometric $\alpha\in D$ determines the congruence $\alpha^{-1}(0)$ on the semigroup $X$. Let $X_\alpha$ be the quotient semigroup of $X$ by the congruence $\alpha^{-1}(0)$, and $\pi_\alpha:X\to X_\alpha$ be the quotient homomorphism. The subinvariant pseudometric $\alpha$ induces a subinvariant metric $\check\alpha$ on $X_\alpha$ such that $\check\alpha(\pi_\alpha(x),\pi_\alpha(y))=\alpha(x,y)$ for any $x,y\in X$.
The subinvariance of the metric $\check\alpha$ implies that the semigroup operation $X_\alpha\times X_\alpha\to X_\alpha$ is uniformly continuous with respect to the metric $\check \alpha$ and hence can be extended to a semigroup operation on the completion $\overline X_\alpha$ of the metric space $(X_\alpha,\check\alpha)$. The absolute $\C$-closedness of the semigroup $X$ implies that $\pi_\alpha[X]=X_\alpha$ is closed in its completion $\overline X_\alpha$, which means that the metric $\check\alpha$ on $X_\alpha$ is complete. The power-Lindel\"of property of the topological space $X$ implies that Lindel\"of property and the separability of the metric space $X_\alpha$.

By Lemma~\ref{l1}, the absolutely $\C$-closed semigroup $X_\alpha$ is zero-closed and by Theorem~\ref{t:MA}(1), the semigroup $X_\alpha$ (admitting the separable complete subinvariant metric $\check\alpha$) is polybounded.

Given two pseudometrics $\alpha,\beta\in D$ we write $\alpha\preceq\beta$ if there exists a positive real number $C$ such that $\alpha(x,y)\le C\cdot\beta(x,y)$ for all $x,y\in X$. This inequality implies the inclusion $\beta^{-1}(0)\subseteq\alpha^{-1}(0)$ which allows us to define a (unique) homomorphism $\pi^\beta_\alpha:X_\beta\to X_\alpha$ such that $\pi_\alpha=\pi^\beta_\alpha\circ\pi_\beta$. The definition of the relation $\alpha\preceq\beta$ implies that the homomorphism $\pi^\beta_\alpha:X_\beta\to X_\alpha$ is Lipschitz and hence continuous. Repeating the argument from the proof of Lemma~\ref{l:MA2}, we can prove that the directed poset $D$ is $\sigma$-directed.

Therefore, $\Sigma=(X_\alpha,\pi_\alpha^\beta:\alpha\preceq\beta\in D)$ is a $\sigma$-directed spectrum of polybounded topological semigroups. The continuous homomorphisms $\pi_\alpha:X\to X_\alpha$, $\alpha\in D$, determine a continuous homomorphism $\pi:X\to\lim\Sigma$, $\pi:x\mapsto(\pi_\alpha(x))_{\alpha\in D}$. Since the topology $\tau$ on $X$ is subinvariant, the homomorphism $\pi$ is injective. The absolute $\C$-closedness of $X$ implies that the image $\pi[X]$ is closed in $\lim\Sigma$. Repeating the argument from the proof of Lemma~\ref{l:MA2}, we can show that $\pi[X]=\lim\Sigma$ and hence $\pi:X\to\lim\Sigma$ is a continuous  isomorphism. Then the space $\lim\Sigma$ is power-Lindel\"of being a continuous image of the power-Lindel\"of space $X$. By Theorem~\ref{t:lim}, the semigroup $\lim\Sigma$ is polybounded and so is its isomorphic copy $X$.


\begin{thebibliography}{}

\bibitem{Alfsen} E.M.~Alfsen, {\em A simplified constructive proof of the existence and uniqueness of Haar measure}, Math. Scand. {\bf 12} (1963), 106--116.

\bibitem{Ban} T.~Banakh, {\em Categorically closed topological groups}, Axioms {\bf 6}:3 (2017), 23.

\bibitem{BanS} T.~Banakh, {\em A   non-polybounded non-topologizable absolutely closed  $36$-Shelah group}, preprint.

\bibitem{CCUS} T.~Banakh, {\em Categorically closed unipotent semigroups}, preprint ({\tt arxiv.org/abs/2208.00072}).

\bibitem{ICT1S}  T.~Banakh, {\em Injectively and absolutely $\mathsf{T_{\!1}S}$-closed semigroups}, preprint ({\tt arxiv.org/abs/2208.00074}).

\bibitem{ICT2S} T.~Banakh, {\em Injectively $\C$-closed commutative semigroups}, preprint ({\tt arxiv.org/abs/2208.13050}).

\bibitem{BB} T.~Banakh, S.~Bardyla, {\em Characterizing categorically closed commutative semigroups}, Journal of Algebra {\bf 591} (2022), 84--110.

\bibitem{CCCS} T.~Banakh, S.~Bardyla, {\em Categorically closed countable semigroups}, preprint, ({\tt arxiv.org/abs/2111.14154}).

\bibitem{ACS}  T.~Banakh, S.~Bardyla, {\em Absolutely  closed semigroups}, preprint ({\tt arxiv.org/abs/2207.12778}).

\bibitem{BRZ} T.~Banakh, A.~Ravsky, L.~Zdomskyy, {\em Cardinal characteristics of the family of Arbault sets in compact Abelian groups}, preprint.

\bibitem{BR} T.~Banakh, A.~Rega, {\em The polyboundedness number of a semigroup}, in preparation.

\bibitem{BS} T.~Bartoszy\'nski, S.~Shelah, {\em Closed measure zero sets}, Ann. Pure Appl. Logic {\bf 58}:2 (1992), 93--110.

\bibitem{Blass} A. Blass, {\em Combinatorial cardinal characteristics of the continuum}, in: Handbook of set theory, Springer, Dordrecht (2010),
395--489.

\bibitem{Chig} A.~Chigogidze, {\em Inverse Spectra}, North-Holland Publishing Co., Amsterdam, 1996. 

\bibitem{Eng} R.~Engelking, \emph{General Topology}, 2nd ed., Heldermann,
 Berlin, 1989.

\bibitem{HS} N.~Hindman, D.~Strauss, {\em Algebra in the Stone-\v Cech compactification. Theory and applications}, De Gruyter Textbook. Walter de Gruyter \& Co., Berlin, 2012.

\bibitem{Howie} J.~Howie, {\em Fundamentals of Semigroup Theory}, Clarendon Press, Oxford, 1995.

\bibitem{Ke} A.~Kechris, {\em Classical Descriptive Set Theory}, Springer-Verlag, New York, 1995.


\bibitem{Strom} K.~Stromberg, {\em An elementary proof of Steinhaus's theorem}, Proc. Amer. Math. Soc. {\bf 36} (1972), 308.
\end{thebibliography}
\end{document}